\documentclass[11pt]{amsart}
\usepackage{amsmath,amsfonts,mathrsfs}

\newif\ifdraft
\draftfalse
\usepackage{amsmath,amsfonts,amsthm,mathrsfs}
\usepackage{amssymb}

\usepackage[unicode,bookmarks]{hyperref}
\usepackage[usenames,dvipsnames]{xcolor}
\hypersetup{colorlinks=true,citecolor=NavyBlue,linkcolor=Brown,urlcolor=Orange}

\usepackage[alphabetic,initials]{amsrefs}

\usepackage{enumitem}

\usepackage{chngcntr}

\ifdraft
\usepackage[notcite,notref,color]{showkeys}

\definecolor{labelkey}{gray}{0.5}
\fi

\usepackage{tikz}

\usepackage{tikz-cd}

\usetikzlibrary{matrix,arrows}
\newlength{\myarrowsize} 

\pgfarrowsdeclare{cmto}{cmto}{
	\pgfsetdash{}{0pt} 
	\pgfsetbeveljoin 
	\pgfsetroundcap 
	\setlength{\myarrowsize}{0.6pt}
	\addtolength{\myarrowsize}{.5\pgflinewidth}
	\pgfarrowsleftextend{-4\myarrowsize-.5\pgflinewidth} 
	\pgfarrowsrightextend{.8\pgflinewidth}
}{
	\setlength{\myarrowsize}{0.6pt} 
  	\addtolength{\myarrowsize}{.5\pgflinewidth}  
	\pgfsetlinewidth{0.5\pgflinewidth}
	\pgfsetroundjoin
	\pgfpathmoveto{\pgfpoint{1.5\pgflinewidth}{0}}
	\pgfpatharc{-109}{-170}{4\myarrowsize}
	\pgfpatharc{10}{189}{0.58\pgflinewidth and 0.2\pgflinewidth}
	\pgfpatharc{-170}{-115}{4\myarrowsize+\pgflinewidth}
	\pgfpathclose
	\pgfusepathqfillstroke
	\pgfpathmoveto{\pgfpoint{1.5\pgflinewidth}{0}}
	\pgfpatharc{109}{170}{4\myarrowsize}
	\pgfpatharc{-10}{-189}{0.58\pgflinewidth and 0.2\pgflinewidth}
	\pgfpatharc{170}{115}{4\myarrowsize+\pgflinewidth}
	\pgfpathclose
	\pgfusepathqfillstroke
	\pgfsetlinewidth{2\pgflinewidth}
}

\pgfarrowsdeclare{cmonto}{cmonto}{
	\pgfsetdash{}{0pt} 
	\pgfsetbeveljoin 
	\pgfsetroundcap 
	\setlength{\myarrowsize}{0.6pt}
	\addtolength{\myarrowsize}{.5\pgflinewidth}
	\pgfarrowsleftextend{-4\myarrowsize-.5\pgflinewidth} 
	\pgfarrowsrightextend{.8\pgflinewidth}
}{
	\setlength{\myarrowsize}{0.6pt} 
  	\addtolength{\myarrowsize}{.5\pgflinewidth}  
	\pgfsetlinewidth{0.5\pgflinewidth}
	\pgfsetroundjoin
	\pgfpathmoveto{\pgfpoint{1.5\pgflinewidth}{0}}
	\pgfpatharc{-109}{-170}{4\myarrowsize}
	\pgfpatharc{10}{189}{0.58\pgflinewidth and 0.2\pgflinewidth}
	\pgfpatharc{-170}{-115}{4\myarrowsize+\pgflinewidth}
	\pgfpathclose
	\pgfusepathqfillstroke
	\pgfpathmoveto{\pgfpoint{1.5\pgflinewidth}{0}}
	\pgfpatharc{109}{170}{4\myarrowsize}
	\pgfpatharc{-10}{-189}{0.58\pgflinewidth and 0.2\pgflinewidth}
	\pgfpatharc{170}{115}{4\myarrowsize+\pgflinewidth}
	\pgfpathclose
	\pgfusepathqfillstroke
	\pgfpathmoveto{\pgfpoint{1.5\pgflinewidth-0.3em}{0}}
	\pgfpatharc{-109}{-170}{4\myarrowsize}
	\pgfpatharc{10}{189}{0.58\pgflinewidth and 0.2\pgflinewidth}
	\pgfpatharc{-170}{-115}{4\myarrowsize+\pgflinewidth}
	\pgfpathclose
	\pgfusepathqfillstroke
	\pgfpathmoveto{\pgfpoint{1.5\pgflinewidth-0.3em}{0}}
	\pgfpatharc{109}{170}{4\myarrowsize}
	\pgfpatharc{-10}{-189}{0.58\pgflinewidth and 0.2\pgflinewidth}
	\pgfpatharc{170}{115}{4\myarrowsize+\pgflinewidth}
	\pgfpathclose
	\pgfusepathqfillstroke
	\pgfsetlinewidth{2\pgflinewidth}
}

\pgfarrowsdeclare{cmhook}{cmhook}{
	\pgfsetdash{}{0pt} 
	\pgfsetbeveljoin 
	\pgfsetroundcap 
	\setlength{\myarrowsize}{0.6pt}
	\addtolength{\myarrowsize}{.5\pgflinewidth}
	\pgfarrowsleftextend{-4\myarrowsize-.5\pgflinewidth} 
	\pgfarrowsrightextend{.8\pgflinewidth}
}{
	\setlength{\myarrowsize}{0.6pt} 
  	\addtolength{\myarrowsize}{.5\pgflinewidth}  
 	\pgfsetdash{}{0pt}
	\pgfsetroundcap
	\pgfpathmoveto{\pgfqpoint{0pt}{-4.667\pgflinewidth}}
	\pgfpathcurveto
    {\pgfqpoint{4\pgflinewidth}{-4.667\pgflinewidth}}
    {\pgfqpoint{4\pgflinewidth}{0pt}}
    {\pgfpointorigin}
	\pgfusepathqstroke
}

\newenvironment{diagram*}[2]{%
\[%
\begin{tikzpicture}[>=cmto,baseline=(current bounding box.center),%
	to/.style={->,font=\scriptsize,cap=round},%
	into/.style={cmhook->,font=\scriptsize,cap=round},%
	onto/.style={-cmonto,font=\scriptsize,cap=round},%
	math/.style={matrix of math nodes, row sep=#2, column sep=#1,%
		text height=1.5ex, text depth=0.25ex}]%
}{%
\end{tikzpicture}%
\]%
\ignorespacesafterend%
}

\newcommand{\CC}{\mathbb{C}}

\newcommand{\PP}{\mathbb{P}}

\newcommand*{\sheafhom}{\mathscr{H}\kern -.5pt om}

\def\overbar#1#2#3{{%
	\setbox0=\hbox{\displaystyle{#1}}%
	\dimen0=\wd0
	\advance\dimen0 by -#2 
	\vbox {\nointerlineskip \moveright #3 \vbox{\hrule height 0.3pt width \dimen0}%
		\nointerlineskip \vskip 1.5pt \box0}%
}}

\makeatletter
\let\@@seccntformat\@seccntformat
\renewcommand*{\@seccntformat}[1]{%
  \expandafter\ifx\csname @seccntformat@#1\endcsname\relax
    \expandafter\@@seccntformat
  \else
    \expandafter
      \csname @seccntformat@#1\expandafter\endcsname
  \fi
    {#1}%
}
\newcommand*{\@seccntformat@subsection}[1]{%
  \textbf{\csname the#1\endcsname.}
}
\makeatother

\makeatletter
\let\@paragraph\paragraph
\renewcommand*{\paragraph}[1]{%
	\vspace{0.3\baselineskip}%
	\@paragraph{\textit{#1}}%
}
\makeatother

\counterwithin{equation}{subsection}
\counterwithout{subsection}{section}
\counterwithin{figure}{subsection}

\newtheorem{theorem}[equation]{Theorem}
\newtheorem*{theorem*}{Theorem}
\newtheorem{lemma}[equation]{Lemma}
\newtheorem*{lemma*}{Lemma}

\newtheorem*{proposition*}{Proposition}
\newtheorem{conjecture}[equation]{Conjecture}

\theoremstyle{definition}

\newtheorem*{definition*}{Definition}
\newtheorem{remark}[equation]{Remark}

\newtheorem*{example*}{Example}
\newtheorem*{problem*}{Problem}

\theoremstyle{plain}

\newcommand{\theoremref}[1]{\hyperref[#1]{Theorem~\ref*{#1}}}
\newcommand{\lemmaref}[1]{\hyperref[#1]{Lemma~\ref*{#1}}}
\newcommand{\definitionref}[1]{\hyperref[#1]{Definition~\ref*{#1}}}
\newcommand{\propositionref}[1]{\hyperref[#1]{Proposition~\ref*{#1}}}
\newcommand{\conjectureref}[1]{\hyperref[#1]{Conjecture~\ref*{#1}}}
\newcommand{\corollaryref}[1]{\hyperref[#1]{Corollary~\ref*{#1}}}
\newcommand{\exampleref}[1]{\hyperref[#1]{Example~\ref*{#1}}}

\makeatletter
\let\old@caption\caption
\renewcommand*{\caption}[1]{%
	\setcounter{figure}{\value{equation}}%
	\stepcounter{equation}%
	\old@caption{#1}\relax%
}
\makeatother

\newcounter{intro}

\newtheorem{intro-conjecture}[intro]{Conjecture}
\newtheorem{intro-corollary}[intro]{Corollary}
\newtheorem{intro-theorem}[intro]{Theorem}

\newcommand{\parref}[1]{\hyperref[#1]{\S\ref*{#1}}}

\makeatletter
\newcommand*\if@single[3]{%
  \setbox0\hbox{${\mathaccent"0362{#1}}^H$}%
  \setbox2\hbox{${\mathaccent"0362{\kern0pt#1}}^H$}%
  \ifdim\ht0=\ht2 #3\else #2\fi
  }
\newcommand*\rel@kern[1]{\kern#1\dimexpr\macc@kerna}
\newcommand*\widebar[1]{\@ifnextchar^{{\wide@bar{#1}{0}}}{\wide@bar{#1}{1}}}
\newcommand*\wide@bar[2]{\if@single{#1}{\wide@bar@{#1}{#2}{1}}{\wide@bar@{#1}{#2}{2}}}
\newcommand*\wide@bar@[3]{%
  \begingroup
  \def\mathaccent##1##2{%
    \if#32 \let\macc@nucleus\first@char \fi
    \setbox\z@\hbox{$\macc@style{\macc@nucleus}_{}$}%
    \setbox\tw@\hbox{$\macc@style{\macc@nucleus}{}_{}$}%
    \dimen@\wd\tw@
    \advance\dimen@-\wd\z@
    \divide\dimen@ 3
    \@tempdima\wd\tw@
    \advance\@tempdima-\scriptspace
    \divide\@tempdima 10
    \advance\dimen@-\@tempdima
    \ifdim\dimen@>\z@ \dimen@0pt\fi
    \rel@kern{0.6}\kern-\dimen@
    \if#31
      \overline{\rel@kern{-0.6}\kern\dimen@\macc@nucleus\rel@kern{0.4}\kern\dimen@}%
      \advance\dimen@0.4\dimexpr\macc@kerna
      \let\final@kern#2%
      \ifdim\dimen@<\z@ \let\final@kern1\fi
      \if\final@kern1 \kern-\dimen@\fi
    \else
      \overline{\rel@kern{-0.6}\kern\dimen@#1}%
    \fi
  }%
  \macc@depth\@ne
  \let\math@bgroup\@empty \let\math@egroup\macc@set@skewchar
  \mathsurround\z@ \frozen@everymath{\mathgroup\macc@group\relax}%
  \macc@set@skewchar\relax
  \let\mathaccentV\macc@nested@a
  \if#31
    \macc@nested@a\relax111{#1}%
  \else
    \def\gobble@till@marker##1\endmarker{}%
    \futurelet\first@char\gobble@till@marker#1\endmarker
    \ifcat\noexpand\first@char A\else
      \def\first@char{}%
    \fi
    \macc@nested@a\relax111{\first@char}%
  \fi
  \endgroup
}
\makeatother

\def\CC{{\mathbf C}}

\def\PP{{\mathbf P}}

\newtheorem*{thmA'}{Theorem~A$^\prime$}

\begin{document}

\title{Conjectures on the Kodaira dimension}

\author[M.~Popa]{Mihnea~Popa}
\address{Department of Mathematics, Harvard University, 
1 Oxford Street, Cambridge, MA 02138, USA} 
\email{{\tt mpopa@math.harvard.edu}}

\thanks{MP was partially supported by NSF grant DMS-2040378.}
\date{\today}

\maketitle

\subsection{Introduction}
In this note I propose a few conjectures on the behavior of the Kodaira dimension under morphisms of smooth complex varieties. The statements are of a rather different flavor than the well-known subadditivity conjectured by Iitaka; in some sense they complement it with ideas inspired by the study of the hyperbolicity of 
parameter spaces. Even though there is one statement that implies them all in \S3, I will first discuss in \S2 an intermediate conjecture that is already of substantial interest
and provides a ``superadditivity" counterpart to subaddivity. Concretely, here is essentially the main case of Conjecture \ref{conj3} in the text: if $X$ and $Y$ are smooth projective varieties, and $f\colon X \to Y$ is an algebraic fiber space with general fiber $F$, which is smooth away from a simple normal crossing divisor $D\subset Y$, then we have
$$\kappa (F)  + \kappa \big(\omega_Y (D) \big)\ge \kappa (X).$$
In \S4 I will also state a weaker conjecture on the behavior of Kodaira \emph{co}dimension, which is potentially more accessible. 

I was inspired to look for such statements while thinking about possible extensions of results for families over abelian varieties in \cite{PS1}. Various cases follow from existing results and techniques, while others have been established since the first announcement 
and are listed here.  Moreover, it now looks to be the case that, just as with Iitaka's conjecture, they are implied by the main conjectures of the minimal model program.

All the varieties in this note are defined over $\CC$.

\noindent
{\bf Acknowledgements.} I thank C. Hacon, J. Koll\'ar, F. Meng, S. Mori , S. G. Park, C. Schnell and Z. Zhuang for comments and for answering my questions.

\subsection{A superadditivity conjecture for morphisms between smooth projective varieties}
Recall that Iitaka's  $C_{n,m}$ conjecture predicts that for an algebraic fiber space $f \colon X \to Y$ of smooth projective varieties  (meaning $f$ is a surjective morphism, with connected fibers), we have subadditivity for the Kodaira dimension, i.e.
$$\kappa (X) \ge \kappa (F) + \kappa (Y).$$
For surveys on this conjecture, see \cite{Mori} and \cite{Fujino}.

Here I start by proposing a complementary ``superadditivity" statement in this same, most common, setting. 

\begin{conjecture}\label{conj3}
Let $f: X \rightarrow Y$ be an algebraic fiber space between smooth projective varieties, and let $V \subseteq Y$ be the open subset over which $f$ is smooth. Then 
 $$\kappa (F) + \kappa (V) \ge  \kappa (X).$$
\end{conjecture}

Recall that the \emph{log Kodaira dimension} $\kappa (V)$ can be defined as follows: after a birational base change we can assume that the complement $D = Y \smallsetminus V$ is a simple normal crossing (SNC) divisor. One then defines
 $$\kappa (V) : = \kappa \big(Y, \omega_Y (D) \big),$$ 
the Iitaka dimension of $\omega_Y (D)$, which is easily checked to be independent of the choice of compactification of $V$ with simple normal crossing boundary.

Conjecture \ref{conj3} will be strengthened later on to one about smooth morphisms between quasi-projective varieties, which 
provides the correction turning the inequality into an equality; see Conjecture \ref{conj-main} below.

\begin{remark}[{\bf Obvious cases}]\label{rmk:obvious}
The  conjecture clearly holds when $\kappa (X) = - \infty$ (so in particular when $\kappa (F) = - \infty$) and when $\kappa (V) = \dim Y$, i.e. $V$ is of log general type. Recall that the Easy Addition lemma, see e.g. \cite[Corollary 2.3(iii)]{Mori}, says that for any algebraic fiber space we have
$$\kappa (F) + \dim Y \ge \kappa (X).$$
\end{remark}

\begin{remark}
It is worth noting that an important class of fiber spaces for which $V$ is of log general type is that of ``moduli" families; more generally, by \cite[Theorem A]{PS2}, relying also on important work on Viehweg's hyperbolicity conjecture in \cite{VZ2}, \cite{CP}, this holds 
for every family with maximal variation such that $F$ admits a good minimal model. The methods used in these works have also been crucial for 
solving other cases of Conjecture \ref{conj3}. 
\end{remark}

\begin{remark}[{\bf Smooth case}]\label{3>4}
In the smooth case, when combined with Iitaka's subadditivity, Conjecture \ref{conj3} leads to the following additivity formula, also generalized later from a different point of view by Conjecture \ref{conj-main}:
\end{remark}

\begin{conjecture}\label{conj4}
 If $f: X \rightarrow Y$ is smooth algbraic fiber space between smooth projective varieties, with general fiber $F$, then
$$\kappa (X) = \kappa (F) + \kappa (Y).$$
\end{conjecture}

\begin{remark}[{\bf Variation and the Kebekus-Kov\'acs conjecture}]
Recall that Viehweg's $C_{n,m}^+$ conjecture states that, when $\kappa (Y) \ge 0$, one has
$$\kappa (X) \ge \kappa (F) + {\rm max}\{ {\rm Var} (f), \kappa (Y)\}.$$ 
(See \cite[\S7]{Mori} for a survey.) It is interesting to note that, when combined with $C_{n,m}^+$,  Conjecture \ref{conj3} implies the Kebekus-Kov\'acs conjecture \cite[Conjecture 1.6]{KK}  for such $Y$, i.e the inequality
$$\kappa  (V) \ge {\rm Var}(f).$$
This last conjecture has been verified when the general fiber of $f$ is of general type by Taji \cite{Taji} and Wei-Wu \cite{WW}; the paper \cite{Taji} also addresses the case when the general fiber of $f$ has a good minimal model.

\end{remark}

\begin{remark}[{\bf Domain of general type}]
It is also amusing to spell out the special case when $X$ is of general type. When $f$ is smooth (i.e. $V = Y$), or when $Y$ is not uniruled, this was proved in \cite{PS3}; the full statement is a consequence of a more general result by Park, see \cite[Corollary 1.6]{Park}.
\end{remark}

\begin{theorem}\label{conj-GT2}\label{GT2}
With the notation in Conjecture \ref{conj3}, if $X$ is of general type, then $V$ is of log general type.
\end{theorem}

\smallskip

The non-obvious cases of Conjecture \ref{conj3} and Conjecture \ref{conj4} are summarized in the next two theorems.

\begin{theorem}\label{thm:conj3}
Conjecture \ref{conj3}  holds when:
\begin{enumerate}
\item $Y$ is an abelian variety, or more generally a variety of maximal Albanese dimension.\footnote{This means that 
$Y$ admits a generically finite (not necessarily surjective) morphism to an abelian variety.}
\item $Y$ is a curve.
\item $f$ is smooth, and either the general fiber of the Iitaka fibration of $Y$ admits a good minimal model,\footnote{More precisely one needs to assume a conjecture of Campana-Peternell, which is in turn a consequence of the existence of good minimal models.} or $Y$ is uniruled. In particular, it holds when $f$ is smooth and $Y$ is a surface or a threefold.  
\item $f$ is smooth, and $\kappa (Y) \ge \dim Y - 3 \ge 0$.
\item $F$ is of general type, so in particular when $X$ is of general type.
\item $F$ has semiample canonical bundle.
\end{enumerate}
\end{theorem}
\begin{proof}
Part (1) is shown in \cite{MP}, using techniques from \cite{LPS}.

Part (2) is clear when $g (Y) \ge 2$. When $Y  = \mathbf{P}^1$ it follows from \cite[Theorem 0.2]{VZ1}, while when $Y$ is elliptic it follows 
from (1).

Part (3) is established in \cite{PS3}; see Theorem C and Corollaries E, F in \emph{loc. cit.}

Part (4) follows by applying Lemma \ref{reduction1} below to a model $g \colon Y' \to Z$ of the Iitaka fibration 
of $Y$, with $Y'$ and $Z$ smooth. Its general fiber $G$ has dimension at most $3$, hence (2) and (3) apply.

Parts (5) and (6) follow from more general results in \cite{PS3}, \cite{Park} for (5), and \cite{Campana} for both, explained in the next section.
\end{proof}

\begin{theorem}\label{thm:conj4}
Conjecture \ref{conj4} holds when
\begin{enumerate}
\item $f$ is a fiber bundle.
\item $Y$ is of general type.
\item $Y$ is an abelian variety, or more generally a variety of maximal Albanese dimension.
\item $Y$ is a curve.
\item $Y$ is a surface.
\item $Y$ is uniruled.
\item $Y$ is a good minimal model with $\kappa (Y) = 0$.
\item $X$ is a good minimal model with $\kappa (X) = 0$.
\item $F$ is of general type.
\item $F$ has semiample canonical bundle.
\item we assume the conjectures of the MMP.
\end{enumerate}
\end{theorem}
\begin{proof} 
Part (1) is one of the original results on the Iitaka conjecture, obtained (in a more general setting) in \cite{NU}.

For (2), (3), (4), and (5), since we know either from Remark \ref{rmk:obvious} or from the previous statement that Conjecture \ref{conj3} is settled in these cases, Conjecture \ref{conj4} follows as in Remark \ref{3>4}, as in all these cases we know that Iitaka's conjecture holds (see \cite{Kawamata2}, \cite{Viehweg1}, \cite{Kawamata3},  \cite{CaP}, \cite{HPS}).

For part (6), see \cite[Proposition G]{PS3}; the main input is the case $Y = \PP^1$, proved in \cite[Theorem 0.2]{VZ1}.

Part (7) is \cite[Theorem H(ii)]{PS3}. On the other hand, if $X$ is a good minimal model with $\kappa (X) = 0$, then so is $Y$ by \cite{TZ}, hence (8) also follows.

Parts (9) and (10) follow from (5) and (6)  in Theorem \ref{thm:conj3} and the fact that Iitaka's conjecture holds when $F$ is of general type by \cite{Kollar}, and when $F$ has semiample canonical bundle by \cite{Kawamata3}.

Part (11) is \cite[Corollary D]{PS3}; a more careful explanation of what exactly is needed is given in \emph{loc. cit.}
\end{proof}

We note that via standard methods one can perform a reduction on the base, which in particular generates further examples (and was already used in Theorem \ref{thm:conj3}(4)). We restrict the discussion to smooth morphisms for simplicity.

\begin{lemma}\label{reduction1}
 Any smooth algebraic fiber space over $Y$ satisfies Conjecture \ref{conj3} if there is a birational morphism $Y' \to Y$, and a nontrivial fibration $g \colon Y' \to Z$ of smooth projective varieties with general fiber $G$, such that $\kappa (Y) = \dim Z + \kappa (G)$\footnote{This happens for instance if $g$ is a model of the Iitaka fibration of $Y$, or if $Z$ is of general type.} and any smooth fiber space over $G$ satisfies Conjecture \ref{conj3}.
\end{lemma}

\begin{proof}
Let $f \colon X \to Y$ be a smooth algebraic fiber space. After a birational base change on $Y$, we may assume that there is a fiber space $g \colon Y \to Z$ to a smooth projective variety, such that its fiber $G$ satisfies Conjecture \ref{conj3}. (Doing the corresponding  birational base change on $X$ is ok, since the fiber product is still a smooth fibration.) Let's denote by $F$ the general fiber of $f$, and by $H$ the general fiber of $h = g\circ f$. We then have that 
$H \to G$ is a smooth fibration with fiber $F$, so by assumption
$$\kappa( F) + \kappa (G) \ge \kappa (H).$$
Adding $\dim Z$ to both sides of the inequality and using the hypothesis, we obtain
$$\kappa(F) + \kappa (Y) \ge \dim Z + \kappa (H) \ge \kappa (X),$$
where the last inequality is the Easy Addition formula applied to $h$.
\end{proof} 

\begin{remark}
One consequence is that in order to establish Conjecture \ref{conj3} for smooth morphisms, it would be enough to show it over all bases $Y$ with 
$\kappa (Y) = - \infty$ or $\kappa (Y) = 0$.  Indeed, when $\kappa (Y) > 0$, we can simply consider a smooth model $g \colon Y' \to Z$ of the Iitaka 
fibration of $Y$. Then either $Y$ is of general type and the statement is clear, or the fiber $G$ of $g$ is positive dimensional with $\kappa (G) = 0$, and the lemma applies.
\end{remark}

\subsection{The most general conjecture}
The strongest proposal about projective morphisms that I would like to make, easily seen to imply all the other conjectures in this note, is the following:
 
 \begin{conjecture}\label{conj-main}
 If $f: U \rightarrow V$ is smooth projective algebraic fiber space between smooth quasi-projective varieties, with general fiber $F$, then
 $$\kappa (U) = \kappa (F) + \kappa (V).$$
 \end{conjecture} 

In other words, in the presence of smooth morphisms, subadditivity in the log version of Iitaka's conjecture should become \emph{additivity}. As noted in the previous section, this is quite open even when $U$ and $V$ are projective. 

There is a mounting body of evidence in favor of this conjecture. To begin with, it is known to hold when $V$ is of log general type, i.e. $\kappa (V) = \dim V$, without any smoothness hypothesis on $f$; see Remark \ref{rmk:lgt} below.  Here are some sample recent results on other cases. The first regards base spaces that compactify to abelian varieties:

\begin{theorem}[{\cite[Theorem A]{MP}}]
Let $f \colon X \to A$ be an algebraic fiber space, with $X$ a smooth projective variety and $A$ an abelian variety. Assume that $f$ is smooth over an open set $V \subseteq A$, and denote $U = f^{-1} (V)$ and the general fiber of $f$ by $F$. Then 
$\kappa (U) = \kappa (V) + \kappa (F)$.
\end{theorem}

Another is that the conjecture holds when $U$ is of log general type; this is a recent theorem of Park. Extending a result shown in \cite{PS3} in the projective case, he proves the following more general fact:

\begin{theorem}[{\cite[Theorem 1.5]{Park}}]
In the situation of Conjecture \ref{conj-main}, assume that $\kappa (F) \ge 0$. Then 
$$V~ {\rm is ~of ~log ~general ~type} \iff  \kappa (U) = \kappa (F) + \dim V.$$ 
In particular, if $U$ is of log general type, then $V$ is of log general type.
\end{theorem}

In \cite[Theorem 1.10]{Park}, Park also completes the proof of the conjecture when $V$ is a curve; due to results in \cite{VZ1} and \cite{MP}, the cases that he needs to establish are when $V$ is $\PP^1$ minus one or two points.

Moreover, the methods of \cite{PS3} show that the conjecture holds when $F$ is of general type, assuming that $\kappa(V) \ge 0$; see Remark 4 in \emph{loc. cit.}  However, more generally and quite importantly, Campana \cite{Campana} has shown that this last assumption can be removed, and furthermore:

\begin{theorem}[{\cite[Theorem 1]{Campana}}]
Conjecture \ref{conj-main} holds when $F$ has semiample  canonical bundle.
\end{theorem}

Using the most general result of \cite{Taji}, still under review, the same proof shows that the conjecture actually holds when $F$ has a good minimal model. 

Let's finish by summarizing this discussion:

\begin{theorem}
Conjecture \ref{conj-main} holds when:
\begin{enumerate}
\item $V$ is of log general type \cite{Kawamata1}, \cite{Maehara}.
\item $V$ compactifies to an abelian variety \cite{MP}.
\item $V$ is a curve \cite{VZ1},\cite{MP}, \cite{Park}.
\item $U$ is of log general type \cite{Park}.
\item $F$ is of general type \cite{PS3}, \cite{Campana}.
\item $F$ has semiample canonical bundle (and possibly, more generally, $F$ has a good minimal model) \cite{Campana}.
\end{enumerate}
\end{theorem}

\noindent
{\bf The log version.}
Even though the logarithmic Kodaira dimension necessarily made an appearance, all the statements above
should be seen as being about the standard case of varieties. Iitaka's conjecture however has more general log analogues; 
see e.g. \cite{Fujino} for a general overview. In this direction, a generalization of Conjecture \ref{conj-main} is as follows:

\begin{conjecture}\label{conj-log}
Let $f \colon X\to Y$ be an algebraic fiber space between smooth projective varieties, and let $E$ be an SNC divisor on $X$ and $D$ an SNC divisor on $Y$ such that ${\rm Supp} (f^*D) \subseteq E$. Assume that $f$ is log-smooth over $V = Y \smallsetminus D$,\footnote{The log-smooth condition means that each stratum of the pair $(X, E)$, including $X$ itself, is smooth over $V$ via $f$.} and let $F$ be a general fiber over a point of $V$.
Then
$$\kappa (X, K_X+ E) = \kappa (Y, K_Y + D) + \kappa (F, K_F + E_F).$$
\end{conjecture}

The reader can specialize this statement to various weaker versions, including those discussed earlier in the standard case. 
For instance, the case of log smooth morphisms to smooth projective varieties corresponds to $D = 0$.

\begin{remark}[{\bf Base of log general type}]\label{rmk:lgt}
Due to a result of Kawamata \cite[Theorem 30]{Kawamata1} for $\kappa (X, K_X + E) \ge 0$, and Maehara \cite[Corollary 2]{Maehara} in general, Conjecture \ref{conj-log} holds when the base is of log general type, i.e. $\kappa (Y, K_Y + D) = \dim Y$.   This is of course a result about Iitaka's conjecture (since the opposite inequality follows from Easy Addition), so it 
does not require the log-smoothness hypothesis.
\end{remark}

\subsection{A conjecture on the Kodaira codimension}\label{scn:codimension}
A weaker but quite interesting consequence of the conjectures in the previous sections can be phrased in terms of the \emph{Kodaira codimension} of a smooth quasi-projective variety $U$, defined in \cite{Mori} as 
$$\kappa c(U) := \dim U - \kappa (U).$$ 

Concretely, at least for smooth fiber spaces, the inequality
$$\kappa (F) + \kappa (V) \ge \kappa (U)$$
predicted as part of Conjecture \ref{conj-main} implies the following: 

\begin{conjecture}\label{conj2}
If $f: U \rightarrow V$ is smooth projective morphism of smooth quasi-projective varieties, then 
$$\kappa c(U) \ge \kappa c (V).$$
\end{conjecture}

I am including this statement because I like its symmetric form, but also because, although weaker, it can still be very useful. For instance, this is what was originally proved in \cite{PS1} when $V$ is an abelian variety. It may also possibly be more approachable in some instances. Note though that while writing the first version of this note, in many significant cases Conjecture \ref{conj2} was known to hold, while the stronger conjectures in the earlier sections were not; however in view of recent work this is not the case anymore.

\begin{remark}\label{rmk-red}
(1) The conjecture holds when $f$ is \'etale, since then $\kappa (U) = \kappa (V)$. When the dimension of the fibers is positive, one may also assume that $f$ is an algebraic fiber space, as then it is not hard to see that in the Stein factorization the finite morphism is in fact \'etale. 

\noindent
(2) The conjecture is not true when $f$ is not smooth. For instance, it is well known that there exist smooth projective surfaces $S$ of general type with $q(S) = 1$, so with a surjective morphism to an elliptic curve; such a morphism cannot be smooth. Or consider any 
base point free pencil on a variety with $\kappa (X) \ge 0$. This induces a morphism $f\colon X \to \mathbf{P}^1$, but such a morphism is never smooth by \cite[Theorem 0.2]{VZ1}. There are many other examples.
\end{remark}

\end{document}


\section*{References}

\begin{biblist}

\bib{Campana}{article}{
    author={Campana, F.},
    title={Kodaira additivity, birational isotriviality and specialness},
    journal={preprint arXiv:2207.05412, v.2},
    date={2022},
}


\bib{CP}{article}{
    author={Campana, F.},
    author={P{\u{a}}un, M.},
    title={Foliations with positive slopes and birational stability of orbifold cotangent bundles},
    journal={Publ. Math. Inst. Hautes \'Etudes Sci.},
    volume={129},
    date={2019},
    pages={1--49},
}

\bib{CaP}{article}{
author={ Cao, J.}, 
author={P\u{a}un, M.},
title={Kodaira dimension of algebraic fiber spaces over abelian varieties}, 
journal={Invent. Math.},
volume={207}, 
date={2017}, 
number={1},
pages={345--387},
}



\bib{Fujino}{article}{
	author={Fujino, Osamu},
	title={Iitaka conjecture -- an introduction},
	journal={SpringerBriefs in Mathematics. Springer, Singapore}, 
	date={2020}, 
	pages={128 pp.},
}


\bib{HPS}{article}{
author={Hacon, C.},
author={Popa, M.},
author={Schnell, C.}, 
title={Algebraic
  fiber spaces over abelian varieties: around a recent theorem by {C}ao and {P}{\u{a}}un},  
journal={Contemporary Mathematics}, 
volume={712},
date={2018}, 
series={Local and global methods in Algebraic Geometry: volume in honor of L. Ein's 60th birthday}, 
pages={143--195},
}



\bib{Kawamata1}{article}{
	author={Kawamata, Yujiro},
	title={Characterization of abelian varieties},
	journal={Compositio Math.}, 
	volume={43},
	date={1981}, 
	number={2},
	pages={253--276},
}

\bib{Kawamata2}{article}{
	author={Kawamata, Y.},
	title={Kodaira dimension of algebraic fiber spaces over curves},
	journal={Invent. Math.}, 
	volume={66},
	date={1982}, 
	number={1},
	pages={51--71},
}

\bib{Kawamata3}{article}{
	author={Kawamata, Yujiro},
	title={Minimal models and the Kodaira dimension of algebraic fiber spaces},
	journal={J. Reine Angew. Math.}, 
	volume={363},
	date={1985}, 
	pages={1--46},
}


\bib{KK}{article}{
	author={Kebekus, Stefan},
	author={Kov\'acs, S\'andor},
	title={Families of canonically polarized varieties over surfaces},
	journal={Invent. Math.}, 
	volume={172},
	date={2008}, 
	number={3},
	pages={657--682},
}

\bib{Kollar}{article}{
 author={Koll{{\'a}}r, J.},
 title={Subadditivity of the {K}odaira dimension: fibers of general type}, 
  journal={Algebraic geometry, {S}endai, 1985, Adv. Stud. Pure
  Math., North-Holland, Amsterdam}, 
  volume={10}, 
  year={1987}, 
  pages={361--398},
  } 
  
  \bib{LPS}{article}{
author={Lombardi, L.},
author={Popa, M.},
author={Schnell, C.}, 
title={Pushforwards of pluricanonical bundles under morphisms to abelian varieties},  
journal={J. Eur. Math. Soc.}, 
volume={22},
date={2020}, 
number={8}, 
pages={2511--2536},
}


\bib{Maehara}{article}{
	author={Maehara, K.},
	title={The weak $1$-positivity of direct image sheaves},
	journal={J. Reine Angew. Math.}, 
	volume={364},
	date={1986}, 
	pages={112--129},
}


\bib{MP}{article}{
	author={Meng, Fanjun},
	author={Popa, Mihnea},
	title={Kodaira dimension of fibrations over abelian varieties},
	journal={preprint arXiv:2111.14165}, 
	date={2021}, 
}

\bib{Mori}{article}{
     author={Mori, S.},
     title={Classification of higher-dimensional varieties},
      journal={Algebraic geometry, {B}owdoin, 1985 ({B}runswick, {M}aine, 1985), Proc.
      Sympos. Pure Math.}, volume={46}, 
      publisher={Amer. Math. Soc., Providence, RI}, 
      date={1987}, 
      pages={269--331}, 
}

\bib{NU}{article}{
	author={Nakamura, I.},
	author={Ueno, K.},
	title={An addition formula for Kodaira dimensions of analytic fibre bundles whose fibre are Moisezon manifolds},
	journal={J. Math. Soc. Japan}, 
	volume={25},
	date={1973}, 
	pages={363--371},
}

\bib{Park}{article}{
	author={Park, Sung Gi},
	title={Logarithmic base change theorem and smooth descent of positivity of log canonical divisor},
	journal={preprint arXiv:2210.02825 }, 
	date={2022}, 
}


\bib{PS1}{article}{
	author={Popa, Mihnea},
	author={Schnell, Christian},
	title={Kodaira dimension and zeros of holomorphic one-forms},
	journal={Ann. of Math.}, 
	volume={179},
	date={2014}, 
	number={3},
	pages={1109--1120},
}

\bib{PS2}{article}{
	author={Popa, Mihnea},
	author={Schnell, C.},
	title={Viehweg's hyperbolicity conjecture for families with maximal variation},
	journal={Invent. Math.}, 
	volume={208},
	date={2017}, 
	number={3},
	pages={677--713},
}

\bib{PS3}{article}{
	author={Popa, Mihnea},
	author={Schnell, Christian},
	title={On the behavior of Kodaira dimension under smooth morphisms},
	journal={preprint arXiv:2202.02825}, 
	date={2022},
}


\bib{Taji}{article}{
	author={Taji, Behrouz},
	title={Birational geometry of smooth families of varieties admitting good minimal models},
	journal={preprint arXiv:2005.01025}, 
	date={2020}, 
}
\bib{TZ}{article}{
	author={Tosatti, Valentino},
	author={Zhang, Yuguang},
	title={Collapsing Hyperk\"ahler manifolds},
	journal={Ann. Sci. \'Ec. Norm. Sup\'er.}, 
	volume={53},
	date={2020}, 
	number={3},
	pages={751--786},
}



\bib{Viehweg1}{article}{
      author={Viehweg, Eckart},
      title={Weak positivity and the additivity of the Kodaira dimension of certain fiber spaces}, 
      journal={Adv. Studies Pure Math.}, 
      volume={1}, 
      date={1983}, 
      pages={329--353},
}      


\bib{VZ1}{article}{
	author={Viehweg, Eckart},
	author={Zuo, Kang},
	title={On the isotriviality of families of projective manifolds over curves},
	journal={J. Algebraic Geom.}, 
	volume={10},
	date={2001}, 
	pages={781--799},
}

\bib{VZ2}{article}{
  author={Viehweg, E.},
  author={Zuo, Kang},
  title={Base spaces of non-isotrivial families of smooth minimal models}, 
  journal={Complex geometry ({G}{\"o}ttingen, 2000), Springer, Berlin}, 
  date={2002},
  pages={279--328},
}

\bib{WW}{article}{
author={Wei, Chuanhao},
	author={Wu, Lei},
	title={Isotriviality of smooth families of varieties of general type},
	journal={preprint arXiv:2001:08360, to appear in Manuscripta Math.}, 
        date={2020}, 
}
  \end{biblist}
\end{document}